\documentclass[a4paper,11pt]{amsart}
\usepackage{amsmath,amsfonts,amssymb,latexsym}
\usepackage{mathrsfs}
\usepackage{graphicx,float}
\usepackage[all,cmtip]{xy}
\usepackage{times}
\newtheorem{lemma}{Lemma}[section]
\newtheorem{corollary}[lemma]{Corollary}

\newtheorem{proposition}[lemma]{Proposition}
\newtheorem{remark}[lemma]{Remark}
\newtheorem{theorem}[lemma]{Theorem}
\newcommand{\field}[1]{\mathbb{#1}}
\newcommand{\wt}[1]{\widetilde{#1}}

\newcommand{\C}{\field{C}}

\newcommand{\K}{\field{K}}
\newcommand{\N}{\field{N}}

\newcommand{\R}{\field{R}}

\newcommand{\Rn}{\R^n}

\newcommand{\Cal}[1]{\mathcal{#1}}
\newcommand{\cF}{\Cal{F}}
\newcommand{\bA}{{\bf A}}

\newcommand{\bb}{{\bf b}}

\newcommand{\be}{{\bf e}}
\newcommand{\bff}{{\bf f}}

\newcommand{\bG}{{\bf G}}

\newcommand{\bp}{{\bf p}}

\newcommand{\bP}{{\bf P}}
\newcommand{\bQ}{{\bf Q}}

\newcommand{\bt}{{\bf t}}
\newcommand{\bu}{{\bf u}}
\newcommand{\bv}{{\bf v}}

\newcommand{\bw}{{\bf w}}

\newcommand{\bx}{{\bf x}}
\newcommand{\by}{{\bf y}}

\newcommand{\cA}{\Cal{A}}
\newcommand{\cB}{\Cal{B}}
\newcommand{\cC}{\Cal{C}}

\newcommand{\cE}{\Cal{E}}

\newcommand{\cI}{\Cal{I}}

\newcommand{\cM}{\Cal{M}}
\newcommand{\cO}{\Cal{O}}
\newcommand{\cP}{\Cal{P}}

\newcommand{\cU}{\Cal{U}}
\newcommand{\cV}{\Cal{V}}
\newcommand{\cW}{\Cal{W}}

\newcommand{\clos}{{\rm clos}}

\newcommand{\dd}{{\partial}}

\newcommand{\lbd}{{\lambda}}

\newcommand{\Gm}{\Gamma}

\newcommand{\omg}{\omega}

\newcommand{\rd}{{\rm d}}
\newcommand{\sgm}{\sigma}
\newcommand{\Sgm}{\Sigma}

\newcommand{\Sp}{{\rm Sp}}

\newcommand{\tldby}{{\tilde{\by}}}

\newcommand{\ve}{\varepsilon}
\newcommand{\vp}{\varphi}
\newcommand{\wdg}{\wedge}
\newcommand{\wtbx}{{\widetilde{\bx}}}
\newcommand{\wtby}{{\widetilde{\by}}}

\newcommand{\wtcE}{{\widetilde{\cE}}}

\newcommand{\wtcU}{{\widetilde{\cU}}}
\newcommand{\wtcV}{{\widetilde{\cV}}}
\newcommand{\wtE}{{\widetilde{E}}}
\newcommand{\wtN}{{\widetilde{N}}}

\numberwithin{equation}{section}
\begin{document}
\title[Re-parameterizing and reducing families of Operators]{Re-parameterizing and reducing families of normal operators}
%    Information for author
%\author[N. Dutertre]{Nicolas Dutertre}
\author[V. Grandjean]{Vincent Grandjean}
%\thanks{Research supported by CNPq grant no 150555/2011-3}
%
\address{Departamento de Matem\'atica, Universidade Federal do Cear\'a
(UFC), Campus do Pici, Bloco 914, CEP. 60455-760. Fortaleza-CE,
Brasil}
\email{vgrandjean@mat.ufc.br}
\thanks{The author is most grateful to IMPAN for their working conditions while visiting
Warsaw and working on this paper. This work was also partially supported by CNPq-Brazil 
grant 305614/2015-0}
%
%    General info
%\subjclass[2000]{Primary 14P10, Secondary 57R70}
%
\date{\today}
%
%
%\dedicatory{This paper is dedicated to our authors.}
%
\begin{abstract}
{We present a new demonstration of a generalization of results of Kurdyka \& Paunescu, and of Rainer, 
which are multi-parameters versions of classical theorems of Rellich and Kato about the reduction in 
families of univariate deformations of normal operators over real or complex vector spaces of finite 
dimensions.

Given a  real analytic normal operator $L:F\to F$ over a connected real analytic real or complex 
vector bundle $F$ of finite rank equipped with a fibered metric structure (Euclidean or Hermitian), 
there exists a locally finite composition of blowings-up of the basis manifold $N$, with smooth centers, $\sgm:\wtN \to N$, such 
that at each point $\wtby$ of the source manifold $\wtN$ it is possible 
to find a neighbourhood of $\wtby$ over which exists a real analytic orthonormal/unitary frame in which the pulled-back 
operator $L\circ\sgm :\sgm^*F\to\sgm^*F$ is in reduced form. 
We are working only with the \em eigen-bouquet bundle \em of the operator and 
a free by-product of our proof is the local real 
analyticity of the eigen-values, which in all prior works was a prerequisite step to get
local regular reducing bases.
}
\end{abstract}
\maketitle
%
%
%
%
%
%
%
%
%
%
%
%
%
%
%

%
%
%
%
%
%
%
%
%
%
%
%                ***********************************************************************
%
%
%
%
%
%
%
%
%
\section{Introduction}
Let be a regular (continuous, $C^k$, $C^\infty$, real analytic) family $\bt \to A(\bt) \in Hom(\K^n,\K^n)$ 
of real or complex normal operators. For each parameter $\bt$ the operator $A(\bt)$ is in reduced form
in some orthonormal/unitary basis $f(\bt)$ of $\K^n$. What could be the regularity of the mapping $\bt \to f(\bt) \in Frame(\K^n)\,$?

Rellich shows that the eigen-values, 
of a real analytic family $\R\ni \bt \to A(\bt)$ of symmetric operators over $\R^n$, can be chosen locally 
real analytic at any parameter $\bt_0$, and there exists a real analytic section 
$(\R,\bt_0) \ni\bt \to f(\bt) \in Frame(\R^n)$ consisting of orthonormal eigen-vectors of $A(\bt)$ \cite{Rel1,Rel2}.
Kato investigates similar problems for single complex variable perturbations 
(holomorphic or not) of a given operator over $\C^n$, and gets some regularity 
results for the eigen-values as well as for the family of diagonalizing bases \cite[Chapter 2]{Kat}. 

The multi-parameter problem was solved more recently with the best possible outcome: 
in general a real analytic family of symmetric matrices, $\R^{k(\geq 2)}\ni\bt  \to A(\bt) \in Hom(\R^n,\R^n)$, 
is not locally diagonalizable along a real analytic (or even continuous) orthonormal frame of eigen-vectors. 
Nevertheless, (and principally) as a consequence of their study of roots of real analytic monic hyperbolic 
polynomials, Kurdyka and Paunescu show that \cite{KuPa}, there exists a real analytic surjective mapping $\sgm: N \to \R^k$,  
with connected source manifold $N$, so that the eigen-values and 
a corresponding orthonormal frame of eigen-vectors of the re-parameterized family $N\ni\wt\bt \to A(\sgm(\wt\bt))$
can be chosen locally real analytic. They deduce the analog result for real analytic 
families of anti-symmetric matrices. 
Around the same time, Rainer have started a series of papers about (mostly) roots of regular monic complex 
polynomials \cite{Rai1,Rai2,Rai3,Rai4}, and finds analogues of Kurdyka \& Paunescu result for 
multi-variate quasi-analytic families of monic complex polynomials, which when applied to
a quasi-analytic family of complex normal matrices, provides, up to a quasi-analytic re-parameterization 
of the family, the possibility to choose locally quasi-analytically the eigen-values as well as a 
local unitary frame of eigen-vectors.

The multi-parameters results 
proceed as follows: First, re-parameterize the domain so that 
the eigen-values can be locally chosen real/quasi-analytic; Second,
solve the corresponding eigen-value linear system and arrange it by further 
re-parameterizations if needs be.
Kurdyka \& Paunescu and Rainer results are that strong because 
quasi-analytic functions admit a resolution of singularities \cite{Hir,BiMi1,BiMi2}
or a local uniformization. Indeed their  
re-parameterization mappings are (locally) finite composition of geometrically admissible blowings-up 
\cite{KuPa,Rai3} and of ramifications \cite{Rai2}. 
Most of the work of 
\cite{KuPa,Rai2,Rai3} lies in the local regularization of the eigen-values.
Observe that regularizing quasi-analytic families of complex polynomials usually requires ramifications
\cite{Rai2}, while quasi-analytic families of complex normal matrices avoid them \cite{Rai3}.

\smallskip 
Before starting to describe the content of our paper let us make the following:

\smallskip\noindent
{\bf Observations:} \em 1) Rellich already noticed that regularizing eigen-vectors bases 
is harder to handle than regularizing eigen-values. Indeed, 
the latter is of projective nature, that is insensitive to rescaling (see for instance 
\cite[p. 111]{Kat}). 

\smallskip 
2) Let $\bt\to A(\bt)\in Hom(\K^n,\K^n)$ be a regular (real or quasi-analytic) family of operators.
Let $\bt_0$ be a point of the parameter space (a regular manifold) admitting a neighbourhood $\cV$ over which
exists a regular mapping $\bt \to f(\bt)\in Frame(\K^n)$ consisting of eigen-vectors
of the family $(A(\bt))_{\bt\in\cV}$. 
Evaluating the regular operator $A$ along the regular local frame provides the local regularity of the eigen-values. 
\em

\medskip
Beside our initial motivation to be interested in this problem (see \cite{Gra,BBGM}),
the above observation is the starting point of the material presented here. 

\bigskip
Let $F$ be a real analytic $\K$-vector bundle of finite rank $n$ over a real analytic connected manifold $N$.
We further assume that $F$ is equipped with a real analytic 
fibered Euclidean ($\K = \R$) or Hermitian  $(\K = \C)$ structure. 
A \em characteristic vector of the normal operator $A:\K^n\to\K^n$, \em 
is either an eigen-vector or, when $\K=\R$ and if it is not an eigen-vector, is lying in a real plane invariant by $A$ 
along which the restriction of $A$ is a similitude (a rotation composed with a dilation). 

\medskip\noindent
In order to hint at what we are doing in this note, let us present the following

\smallskip\noindent
{\bf Example.} 
We re-visit \cite[Example 6.1.]{KuPa} with our point of view:
Let $F:=N\times \R^2$ with $N = \R^2$, and let $L:F\to F$ be given by the symmetric matrix
$$
M(x,y) = 
\left[
\begin{array}{cc}
x^2 & xy \\
xy & y^2
\end{array}
\right].
$$
The eigen-values are $0$ and $x^2 + y^2$, and the eigen-spaces are 
$(1,\frac{y}{x})$ and $(1,-\frac{x}{y})$. There is 
a $\bP^1$ of distinct pairs of eigen-spaces accumulating at $(0,0)$.  

\smallskip\noindent
The eigen-bouquet bundle of $L$ is the \em union of the eigen-spaces bundled over $N$, \em that is the subset:
$$
\{((x,y);(X,Y))\in F \,:\, M(x,y)\cdot(X,Y)\wdg(X,Y) = 0\} = \{ (yY + xX)(yX-xY) = 0\}.
$$
Let $\bx := (x,y)$. The ideal $J_\bx$ of the polynomials over $F_\bx$ vanishing along 
the eigen-bouquet of $L(\bx)$ is reduced, principal and given by 
$$
J_\bx = (yY + xX)(yX - xY).
$$
As expected $J_{(0,0)} = (0)$.  
Blowing-up the origin $(0,0)$ in $N$, the base manifold of the vector bundle $F$, and considering
the chart $\sgm:\bu:=(u,v) \to (u,uv)$ (the other chart
will lead to a similar computation), gives
\begin{center}
$\sgm^*(J_\bx) = u^2(vY - X)(vX+Y) =(u)^2\cdot J_\bu'$ with $J_\bu' = (vY + X)(vX - Y)$
\end{center}
the vanishing locus of $J_\bu'$ is a union of two orthogonal lines
of the fiber $(\sgm^*F)_\bu$, each of which invariant by $(L\circ\sgm)(\bu)$. Those pairs of lines 
"move" analytically in the parameter $\bu\in \wtN$.

\smallskip\noindent
{\bf Remarks about the example.} \em
Controlling point-wise all the eigen-spaces (the eigen-bouquet) is achieved
by means of a single object: \em the family of quadratic ideals $(J_\bx)_{\bx\in N}$. \em This
method did not involve resolving simultaneously the equations of the eigen-spaces, 
neither it needed to do anything about the eigen-values. \em

\medskip
The discriminant locus $D_L$ of an operator $L:F\to F$ is the locus of points of $N$ where the number 
of distinct (complex) eigen-values of $L$ is not locally constant, in other words it is the
discriminant locus of the reduced characteristic polynomial of $L$ (that is the vanishing locus of the  
generalized discriminant of the characteristic polynomial \cite{Whi}). 
A center of blowing-up is \em geometrically admissible \em if it is smooth
and normal crossing with the existing exceptional divisor (see \cite{Hir,BiMi1}).

The main result of this paper is the following:

\medskip\noindent
{\bf Theorem \ref{thm:main-global}.} \em
Let $L:F \to F$ be a real analytic  normal operator over a real analytic $\K$-vector bundle
$F$ of finite rank over a connected real analytic manifold $N$, and equipped with a real analytic 
fiber Euclidean/Hermitian structure. 

\smallskip
There exists a locally finite composition of geometrically admissible blowings-up $\sgm:(\wtN,\wtE) \to (N,D_L)$
such that for any $\tldby \in \wtN$, there exists $\wtcV$ a neighbourhood of $\tldby$
and real analytic vector sub-bundles $R_1,\ldots,R_s$ of $\sgm^*F|_\cV$ such that 

(a) They are pair-wise orthogonal and everywhere in direct sum;

(b) The restriction of $L\circ\sgm$ to each $R_i$ is reduced: either it is a multiple
of identity or $R_i$ is real of real rank 2, and the restriction of $L\circ\sgm$
is fiber-wise a similitude. 
\em

\smallskip
As announced and suggested in the example above, we {\bf neither computed eigen-values 
nor solved any eigen-value linear system}. Our proof does not fix/better (re-parameterize)
the eigen-values of the operator $L$. 

\smallskip
Evaluating $L\circ\sgm$ along 
the sub-bundles $R_i$ gives the following

\medskip\noindent
{\bf Corollary \ref{cor:main-global}.} \em
For any $\tldby \in \wtN$, there exists $\wtcV$ a neighbourhood of $\tldby$ such that 
the eigen-values of $\,\wtbx\to L(\sgm(\wtbx))$ are real analytic.
\em 

\medskip
Finding locally a real analytic frame of eigen-vectors is problematic only at points of the 
 discriminant locus $D_L$ of $L$: different eigen-value functions 
take the same value and the accumulations at a given discriminant point of their corresponding eigen-spaces can be 
quite complicated. The problem here is not only to separate different limits of a sub-bundle
(eigen-space) but do that simultaneously for several. Outside $D_L$ we can define a real analytic 
multi-valued mapping 
$$
\begin{array}{rcl}
N\setminus D_L & \buildrel\Gm\over\longrightarrow & \cup_{k=1}^n \bG(k,F)\\
\bx & \longrightarrow & \{[E_1(\bx)],\ldots,[E_{s_L}(\bx)]\}
\end{array}
$$
where $s_L$ is the number of distinct complex eigen-values outside the discriminant locus $D_L$ 
and $E_1(\bx),\ldots,E_{s_L}(\bx)$ are the eigen-spaces (when diagonalizable), characteristic spaces 
(when not diagonalizable) of $L(\bx)$.
\\
Note that the second part of the proofs of \cite{KuPa,Rai3} locally regularizes the "eigen-spaces" 
mappings one after another, using the local regularity of the eigen-values.

\smallskip\noindent
The proof of Theorem \ref{thm:main-global} uses the family of polynomial ideals $(J_\by)_{\by\in N}$ 
vanishing along the \em bouquet of eigen-spaces \em or \em eigen-bouquet: \em the union of the eigen-spaces 
of $L(\by)$. This family of ideals encodes information about the family of eigen-bouquets as does $(\Gm(\by))_{\by\in N}$.
Any such ideal $J_\by$ is reduced and generated by quadratic polynomials.
Let $J_{\by,2}\subset (S^2F^\vee)_\by$ (the second symmetric power of the dual bundle $F^\vee$ of $F$)
be the subspace of the quadratic polynomials over $F_\by$ generated 
by the quadratic polynomials belonging to $J_\by$. Outside $D_L$ its dimension is constant equal to $d_L$.
To the (vector) bundle-d union $\cup_{\by\in N} J_{\by,2}\subset S^2F^\vee$ we associate the $\cO_N$-coherent module
of its real analytic sections, say $\cA_L$, which is locally free of rank $d_L$. Let $\cM_L := \wedge^{d_L} \cA_L$ 
which is also a $\cO_N$-coherent module with co-support $D_L$. Its ideal of coefficients $\cF_L$ is the maximal Fitting 
ideal of $\cA_L$, that is the $\cO_N$-ideal locally generated by the $d_L\times d_L$-minors of $\cA_L$. 
The ideal $\cF_L$ can be principalized (and monomialized, though unnecessary) by the classical Theorems of 
embedded resolution of singularities by a locally finite composition of geometrically admissible blowings-up 
$\sgm:(\wtN,\wtE) \to (N,D_L)$. In other words $\cO_\wtN$-module $\cM_L' := (\sgm^*\cF_L)^{-1}\cdot\sgm^*\cM_L$ 
(weak transform of $\sgm^*\cM_L)$ is locally free of rank $1$ with empty co-support. The module $\cM_L'$ 
corresponds to a vector sub-bundle of $S^2F^\vee$ of rank $d_L$
which generates the reduced ideals $J_\bx' \subset (S(\sgm^*F)^\vee)_\bx$, which at every point $\bx$ vanishes 
along a bouquet of eigen-subspaces of $L\circ\sgm(\bx)$ (in the diagonalizable case) which are in orthogonal direct sum.
And thus we are done for the diagonalizable case.
The real normal but not symmetric case is deduced from the real symmetric case. 

\medskip
The paper is organized as follows:
Section \ref{section:MP} introduces very quickly some of the  objects to be used in the paper.
Section \ref{section:ANO} recalls elementary facts about normal operators and insists about the real ones 
when not diagonalizable.
Section \ref{section:SVEPA} and Section \ref{section:SLAMA} present very elementary algebraic facts about
bouquets of linear subspaces which are in direct sums and how these properties relates to diagonalization
of an endomorphism and to the ideal of the polynomials vanishing along the bouquet.
Section \ref{section:TKL} presents the Essential Lemma \ref{lem:essential} and its Corollary.
Although mostly technical, this Lemma catches the very nature of the 
structure we are looking for.
The most important work is done in Section \ref{section:RM}, which after studying thoroughly 
some local properties of the (pointed)-eigen-bouquet bundle, provides a proof of the main result 
in the local case, that is when $F$ is trivial under some special hypotheses that we know achievable 
by composition of blowings-up if need be. 
Section \ref{section:GEM} finishes to treat the whole local case, and
Remark \ref{rk:one-dim} shows how little work the 
one-dimensional case requires. Section \ref{section:GC}
deals with the general case presented above and Section \ref{section:RC} ends the real
case. We would like to single out Remark \ref{rmk:invertible-sheaf} to emphasize
what our method can do that the regularization of eigen-values strategy would struggled to achieve.
%
%
%
%
%
%
%
%
%
%
%
%                ***********************************************************************
%
%
%
%
%
%
%
%
%
\section{Miscellaneous preliminaries}\label{section:MP}
$\bullet$
In the rest of this note we will write $\K$ either for $\R$ or for $\C$.

\smallskip\noindent
$\bullet$ A \em $\K$-vector bundle $F$ \em over a topological space $S$
is just that: a "bundle of $\K$-vector spaces" over each point of $S$, namely
$$
F = \cup_{\by \in S} \{\by\}\times F_\by
$$
where each $F_\by$ is a $\K$-vector space.

If $R$ is any subset of $S$, let us denote $F|_R:= \cup_{\by\in R}$ the restriction of $F$
to $R$.

If all fibers $F_\by$ have dimension (over $\K$) $0 \leq r < +\infty$, then we speak of a $\K$-vector bundle or \em rank $r$. \em

Given a $\K$-vector bundle $F$ over $S$ we associate the projective bundle $\bP F:=\cup_{\by\in S}\{\by\}\times \bP F_\by$.
A subset $Z:=\cup_{\by\in S}\{\by\}\times Z_\by$ with ($\emptyset\neq$)$Z_\by\subset\bP F_\by$ is called a \em
sub-bundle of $\bP F$. \em To any such sub-bundle of $\bP$, we can define $Cone(Z) := \cup_{\by\in S}\{\by\}\times Cone(Z_\by)$ 
a subset of $F$ called \em the cone bundle over $Z$, \em where $Cone(Z_\by)$ is the $\K$-cone of $F_\by$
over $Z_\by$. 

\smallskip
Let $S$ be a connected manifold of regularity $\cC$ (continuous, $C^k$ for $k=1,\ldots,\infty$, real analytic).
A real $\cC$-$\K$-vector bundle over $S$ of rank $r$ is a $\K$-vector bundle $F$ over $S$ 
which is a $\cC$-manifold and which is locally $\cC$-trivial:
that is for each $\by\in S$, there exists an open neighbourhood $\cU$ of $\by$ in $S$ such that:

(i) there exists a $\cC$-diffeomorphism $\Phi:F|_\cU \to \cU \times \K^r$;

(ii) the mapping $\Phi$ is such that $\Phi(\bx;\bv) = (\vp(\bx),G(\bx)\bv)$, where $\vp$ is a $\cC$-diffeomorphism
of $\cU$ onto $\vp(\cU)$, and $G(\bx):F_\bx\to \K^r$ is a $\K$-linear isomorphism.

\smallskip
Let $F$ be a $\cC$-$\K$-vector bundle over $N$ and let $\sgm: N'\to N$ be a $\cC$-regular mapping.
We denote $\sgm^*F$ the $\cC$-$\K$-vector bundle over $N'$ obtained by the base change $\sgm$.

\smallskip
Given $F$ a $\cC$-$\K$-regular vector bundle of rank $r$ over the $\cC$-manifold $S$, let $\dot F$ (also denoted $F^\cdot$ or $(F)^\cdot$)
be the \em pointed $\cC$-vector bundle of $F$, \em namely the fiber $\cC$-sub-bundle of $F$ obtained
by removing its zero section $S\times 0$.

\smallskip
Let $F$ a $\cC$-$\K$-regular vector bundle of rank $r$ over the $\cC$-manifold $\cU$. 
For every $0\leq s\leq r$, let $\bG(s,F)$ be the Grassmann fiber $\cC$-bundle of $F$ of the 
$\K$-vector subspaces of $F$ of dimension $s$, in other words $\bG(s,F)_\by := \bG(s,F_\by)$
at any $\by\in S$. 
%%%We will also use $\bP F = \bG(1,F)$ the projective bundle obtained from $F$.

Let $SF^\vee$ be the symmetric algebra vector bundle over $F$, built from $F^\vee$ the dual of $F$,
$$
SF^\vee = (\cU\times\K)\oplus F^\vee \oplus S^2F^\vee\oplus \ldots \oplus S^kF^\vee\oplus \ldots
$$
where $S^kF^\vee$, for $k\geq 2$ is the symmetric bundle of $F $ of order $k$.

\bigskip\noindent
$\bullet$
In the real analytic category we will sometimes use \em regular \em for a subset or a mapping to
mean it is real analytically regular.
%
%
%
%
%
%
%
%
%
%
%
%
%
%
%                ***********************************************************************
%
%
%
%
%
%
%
%
%
%
%
%
\section{About normal operators}\label{section:ANO}
The material presented here is basic, but may be for some aspect of the terminology.
The real part will be used in Section \ref{section:RC}.

\bigskip
Let $F$ be a $\K$-vector space of finite dimension $n$, 
equipped with a scalar product (meaning Hermitian if $\K = \C$), denoted $\langle-,-\rangle$ and
simply $|-|$ for the associated norm.

A $\K$-endomorphism $L\in Hom(F,F)$ is normal, if the adjoint $L^*$ of $L$ commutes with $L$, where 
$L^*$ is uniquely defined by the following condition
$$
\forall \, \bu,\, \bv \, \in F\; \mbox{ we have } \;\langle L\bu,\bv\rangle = \langle \bu,L^*\bv\rangle.
$$
Let $\Sp(L)$ be the (complex) spectrum of the normal operator $L:F\to F $. Thus:
(i) $\lbd \in \Sp(L)$ if and only if $\bar\lbd\in\Sp(L^*)$; and (ii) Let $\lbd \in \Sp(L)$. 
Then $\ker(L - \lbd Id) = \ker(L^* -\bar\lbd Id)$.
Normal endomorphisms and their adjoint which are diagonalizable, are simultaneously so in an 
orthonormal/unitary
basis of eigen-vectors: Any complex normal endomorphism is diagonalizable in the unitary group.

\bigskip
Let $V$ be the real Euclidean space of dimension $n$. Given $L:V\to V$ a normal endomorphism, we recall 
$$
V = R \oplus C
$$ 
where $R$ and $C$ are orthogonal subspaces of $E$ such that:
(i) The subspaces $R$ and $C$ are invariant by $L$;
(ii) The restriction $L|R:R\to R\subset V$ is diagonalizable, once $R$ is not the null space;
(iii) The restriction $L|C:C \to C\subset V$ is an isomorphism of $C$ with no real eigen-values, 
once $C$ is not reduced to the null space. We call $C$ \em the anti-real characteristic space of $L$. \em
Whenever $C$ is not the null space, it decomposes as an orthogonal direct sum
$$
C = C_1 \oplus \ldots \oplus C_t
$$
of vector subspaces $C_1,\ldots,C_t$ of dimension $2$ invariant by $L$, 
such that the restriction $L|C_i$ is a \em similitude: \em a composition of a rotation and a 
dilation.
Any such a plane $C_i$ is called an \em anti-real characteristic plane \em (\em ARCP for short\em )
of $L$. In particular we can write 
\begin{center}
$
L = A + B, \;$ where $\;A:=\frac{L+L^*}{2}\;$ and $\; B = \frac{L-L^*}{2}
$
\end{center}
so that 
\begin{center}
$
A= A^* \;$ and $\; B^* = - B. 
$
\end{center}
The eigen-values of $B$ are pure imaginary complex numbers and square roots of 
those of $B^2 = - BB^*$. 
Any eigen-space $W$ of $BB^*$ is contained in an eigen-space of $A$ and there exist real numbers $a,b$ such that
\begin{center}
$
A|W = a Id_W \; \; {\rm and } \;\; BB^*|W = -b^2 Id_W
$
\end{center}
and thus $a+\sqrt{-1}b$ and $a-\sqrt{-1}b$ are complex eigen-values of (the complexification of) $L$.

\smallskip
 We define the following operators $J$ and the symmetric operator $B_2:=J(B\oplus B)$:

\smallskip
\begin{center}
$
\begin{array}{ccc}
J  :  V\oplus V & \to & V \oplus V \\
\;\;\;\;\;\; \bu\oplus\bv & \to & (-\bv) \oplus \bu
\end{array}
$
\;\;
$
\begin{array}{ccc}
B_2   :  V\oplus V & \to & V \oplus V \\
\;\;\; \;\;\;\; \bu\oplus\bv & \to & (-B\bv) \oplus B\bu
\end{array}
$
\end{center}

We recall
\begin{lemma}\label{lem:plane-invariant} 
1) $b \in \Sp(B_2)$ if and only if $-b\in\Sp (B_2)$ if and only if
$-b^2\in \Sp (BB^*)$.

\smallskip\noindent
2) Let $b \neq 0$ be an eigen value of $B_2$ and let $E_b$ be the corresponding eigen-space, and let 
$\bu\oplus\bv \in E_b\setminus 0\oplus 0$. The following hold true:

(i) $\bu,\bv$ are orthogonal and the plane $\R\bu + \R\bv$ of $V$ is invariant by $B$

(ii) $\bv\oplus\bu \in E_{-b}$;

(iii) $J(\bu\oplus\bv) = (-\bv)\oplus \bu \in E_b$; 

(iv) $JE_b = E_b$.
\end{lemma}

A normal endomorphism $L:F\to F$ is \em reduced \em in the basis $\cB = (\bb_1,\ldots,\bb_n)$ of $F$, 
if it consists of orthonormal/unitary characteristic vectors of $L$, that is 
diagonalizes $L$ if $\K = \C$, and if $\K = \R$ it is a basis corresponding to the orthogonal 
direct sum $F = R\oplus C_1\oplus\ldots\oplus C_s$ such that $L|R$ is diagonal in the basis 
$(\bb_1,\ldots,\bb_r)$ of $R$, and each $C_i$ is an ARCP.
%
%
%
%
%
%
%
%
%
%
%
%
%
%
%
%                ***********************************************************************
%
%
%
%
%
%
%
%
%
%
%
%
\section{Some very elementary polynomial algebra}\label{section:SVEPA}
Let $F$ be a $\K$-vector space of dimension $n$. Let $E_1,\ldots,E_s$ be pairwise distinct non-trivial vector subspaces of $F$,
that is each $E_i$ is neither the null-space nor is $F$.

\smallskip
The union $E :=E_1\cup\ldots\cup E_s$ is called a \em bouquet of subspaces \em of $F$.

\medskip
Let $SF^\vee$ be the $\K$-algebra of polynomials over $F$ (isomorphic to $\K[X_1,\ldots,X_n]$
$:=\K[X]$) 
and let $S^2F^\vee$ be the vector space of quadratic polynomials of $F$ (which
we identify with $H_2[X]\subset\K[X]$, the $\K$-vector subspace of $\K[X]$
consisting only of quadratic polynomials).

\bigskip
Let us consider the subset 
$$
S_k^n:=\{(e_1,\ldots,e_k) \in \{1,\ldots,n\}^k \;:\; e_1\leq \ldots \leq e_k \; \mbox{ and } \;\sum_1^k e_i = n\}.
$$
Given $\be \in S_k^n$ is associated a unique partition of $\{1,\ldots,n\}$, namely 
$n_\be = \{n_\be^1,\ldots,n_\be^k\}$
for 
$$
n_\be^i := \{j \; : \; 1+e_1 + \ldots e_{i-1} \leq j \leq e_1+ \ldots +  e_i\}
$$ 
with the convention $e_0 := 0$.

\medskip
Let $\bx =(x_1,\ldots,x_n)$ be a linear system of coordinates in $F$, and let 
$E_1,\ldots,E_s$ be $s\geq 2$ coordinates subspaces which are in direct sum
$$
E_1\oplus \ldots \oplus E_s = F.
$$
Let $e_i$ be the dimension of $E_i$. 
Denoting $e_0:=0$, we can assume that 
$$
E_i := \{(0,\ldots,0,x_{1+e_1+\ldots+e_{i-1}}, \ldots, x_{e_1+\ldots+e_i},0,\ldots,0)\in F: x_k\in\K\}
$$

Let $J_i$ be the ideal of $SF^\vee$ of polynomials vanishing onto $E_i$, that is 
$J_i = (X_j)_{j\notin n_\be^i}$.

\medskip
Let $E :=\cup_{i=1}^s E_i$ and let $J$ be the ideal of $SF^\vee$, 
vanishing along the bouquet $E$. Let $\be\in S_k^n$ corresponding to the dimensions of $E_1,\ldots,E_k$
(after permutation is necessary). We do have the elementary 
\begin{lemma}\label{lem:elementary}
(i) Let $J = \cap_{i=1}^s J_i$.

\smallskip\noindent
(ii) $J$ is minimally generated by all the products $x_ix_j$, with $i<j$ and 
$(i,j)\in \cup_{1\leq l<m\leq s} \, n_\be^l \times n_\be^m$ 

\smallskip\noindent
(iii)  We have 
\begin{equation}\label{eq:dimJ}
\dim J\cap S^2F^\vee = \sum_{j=1}^{s-1} e_j(e_{j+1}+\ldots + e_s) \,
%= \sum_{1\leq i<j\leq s}e_i e_j
= \,\sum_{1\leq i<j\leq s}\; \dim E_i\cdot \dim E_j \;\leq \frac{n(n-1)}{2}.
\end{equation}
\end{lemma}
\begin{proof}
Each ideal $J_i$ is generated by $(X_k)_{k\notin n_\be^i}$.
An induction on $s$ (or on $n-s$) will do.
\end{proof}
\begin{corollary}\label{cor:elementary}
Let be two direct sum decompositions of $\K^n$ of the following form:
$$
\K^n = E_1\oplus\ldots\oplus E_r\oplus F_1\oplus \ldots \oplus F_s
\;\;
{\rm and }
\;\;
\K^n = E_1'\oplus\ldots\oplus E_r'\oplus G_1\oplus \ldots \oplus G_t
$$
such that: 
(i) $r,s,t \geq 1$ with $t<s$.
\\
(ii) $\dim E_i = \dim E_i'$;
\\
(iii) there exist a $t$-tuple $(d_1,\ldots,d_t) \in \N_{\geq 1}^t$ such that $d_1 + \ldots + d_t = s$ and 
a partition $\sqcup_{k=1}^t \cA_k$ of $\{1,\ldots,s\}$ with $|\cA_k| = d_k\geq 1$ 
such that for each $k=1,\ldots,t,$ we have 
$$
\dim G_k = \sum_{j\in \cA_k} \dim F_j.
$$

Let $J$ be the ideal of $SF^\vee$ of polynomials vanishing over $(\cup_iE_i)\cup(\cup_jF_j)$ and 
$J'$ be the ideal of polynomials vanishing over $(\cup_iE_i')\cup(\cup_k G_k)$.
Then 
$$
\dim J'\cap S^2F^\vee < \dim J \cap S^2F^\vee
$$
\end{corollary}
\begin{proof}
It is an elementary computation from Equation (\ref{eq:dimJ}).
\end{proof}
If $K$ is any ideal of $SF^\vee$, 
let $K_2:=K\cap S^2F^\vee$ be the vector subspace of the quadratic polynomials contained
in $K$. Let $d_2(K)$ be the dimension of $K_2$.
We can now state the following
\begin{lemma}\label{lem:elementary-equal}
Let $J$ as in Lemma \ref{lem:elementary}. Let $K$ be an ideal of $SF^\vee$ 
generated by quadratic polynomials and such that $d_2(K) = d_2(J)$.
If $K$ vanishes along $E$ then $K=J$.
\end{lemma}
\begin{proof}
The ideal $J$ is generated by $d_2(J)$ quadratic monomials.
Since it is reduced, it contains $K$, so that any quadratic generators of 
$K$ is a linear combination of the monomials generating $J$, since $K_2 \subset J_2$.
But they have the same dimension so $J_2 = K_2$ and then $K = J$.
\end{proof}
The following Lemma will be quite useful later in the paper.
\begin{lemma}\label{lem:elementary-coincides}
Let $J$ as in Lemma \ref{lem:elementary}. Let $K$ be an ideal of $SF^\vee$ 
generated by quadratic polynomials.
If the vanishing locus $V(K)$ of $K$ is exactly
the vanishing locus $V(J)$ of $J$, then $K=J$.  
\end{lemma}
\begin{proof}
Since the ideal $J$ is reduced, we deduce that $K \subset J$.
It is sufficient to show that $K\cap S^2F^\vee= J\cap S^2F^\vee$.
In particular $K_2$ is a vector subspace of $J_2$. 
Assume that $J_2 = K_2 \oplus L_2$ for $L_2$ a subspace of $S^2F^\vee$ of positive
dimension. If $L$ denotes the ideal of $SF^\vee$ generated by $L_2$ we find
$$
V(J) = V(K) \cap V(L)
$$
implying that $V(L) \cap V(J) = V(K)$, by hypothesis.
If the latter assumption were  true it would contradict the minimality 
of the generators of $J$ explicitly listed  above. 
\end{proof}
Our interest in these elementary results lies the following
\begin{proposition}\label{prop:diagonal-equivalent-conditions}
Let $A \in M_n(\K)$ be a $n\times n$ matrix with $\K$-entries. 

\medskip
1) Assume that $A$ is not
a multiple of the identity. Let $E_1,\ldots,E_s$ be the eigen-~spaces of $A$, and let $J$ 
be the ideal of polynomials of $\K[X]$ vanishing over $\cup_{i=1}^s E_i$. 

\smallskip
The following conditions are equivalent:

(i) $A$ is diagonalizable; 

(ii) $J$ is generated by quadratic polynomials;

(iii) The ideal generated by $J_2:=J\cap H_2[X]$ is equal to $J$;

(iv) The ideal generated by $J_2$ is reduced.

\medskip
2) If $A$ is a multiple of the identity, then the ideal $J$ is the null ideal.
\end{proposition}
\begin{proof}
Obvious using the previous two Lemmas.
\end{proof}
%
%
%
%
%
%
%
%
%
%
%
%                ***********************************************************************
%
%
%
%
%
%
%
%
%
\section{Some linear and multi-linear algebra}\label{section:SLAMA}
Let $F$ be a $\K$-vector space of dimension $n$, and let $F^\vee$ be its dual.
We assume $F$ is equipped with a Euclidean ($\K=\R$) or a Hermitian product ($\K =\C$).
\\
Let $L:F \to F$ be a normal operator which we assume to be diagonalizable.
Further assuming that $L$ is not a multiple of the identity map, 
let $E_1,\ldots,E_s$ be all the eigen-vector subspaces of $L$, with $s\geq 2$.
We consider the mapping 
$$
\begin{array}{rcl}
Q_L \, : \, F & \to & \wdg^2F \\
 \xi & \to & L\xi\wdg \xi
\end{array}
$$
Since $\wedge^2F$ embeds into $S^2F^\vee$ as the skew-symmetric forms, 
we can also consider that $Q_L$ takes values in $(S^2F^\vee)^\frac{n(n-1)}{2}$.

\smallskip
Let us define 
$$
EV_L:=\{\xi\in F\setminus 0: Q_L(\xi) = 0\} = (E_1)^\cdot \cup \ldots \cup (E_s)^\cdot.
$$ 
the union of the pointed eigen-vector subspaces of $L$, 
which we rename as the (\em{pointed}\em) \em eigen-bouquet of $L$. \em 
It is a smooth submanifold of the pointed vector space 
$F^\cdot$. Since it is invariant by the action of $\K^*$, it gives rise
to a projective submanifold of $\bP F$ with $s$ connected/irreducible components. 

\smallskip
Let $DQ_L: TF^\cdot =F^\cdot\times F \to T(\wdg^2 F)$ be the differential mapping 
of $Q_L$. Looking at $DQ_L$ as a mapping $F^\cdot \to Hom(F,\wdg^2 F)$
with linear coefficients, it is a regular section of some regular vector bundle over $F^\cdot$.

Let  $\cB = (\bb_1,\ldots,\bb_n)$ be an orthonormal/unitary basis of $F$, and 
let $M$ be the matrix of $L$ in this basis. Let $\bb_{(i,j)} := \bb_i\wdg\bb_j$.
We write 
$$
L\bv \wdg \bv = M\bv \wdg \bv = : \, \bQ(\bv) \, = \sum_{1\leq i<j\leq n}\, Q_{i,j}(\bv) \, \bb_{(i,j)}.
$$
with $Q_{i,j}(\bv):=
%\langle A_{i,j}\bv,\bv\rangle_0 = 
(M\bv)_i v_j - (M\bv)_j v_i$. We do see that 
$$
D_\bw \bQ 
%= 2\sum_{1\leq i<j\leq n}
%\langle A_{i,j}\bw,\rd\bv\rangle_0 \;\bb_{(i,j)} 
=  \sum_{1\leq i<j\leq n} \omg_{(i,j)}(\bQ,\bw) \; \bb_{(i,j)}
$$
where $\omg_{(i,j)} (\bQ,\bw)$ is a $\K$-linear form over $T_\bw F$ with linear (in $\bw$) coefficients in $F$.

A non-zero vector $\bv$ lies in $EV_L$ if and only if $Q_{i,j}(\bv) = 0$ for all $i<j$. Thus we deduce
that 
$$
D_\bv \bQ \cdot \bv = 0
$$

Assume that the basis $\cB$ is such that the matrix $M$ is diagonal with diagonal coefficients $a_1,\ldots,a_n$. 
Thus there exist $i<j$ with $a_i \neq a_j$. 
We obtain
$$
\bQ (\bv) := M\bv \wdg \bv = \sum_{i<j} (a_i-a_j) v_iv_j \;\bb_{(i,j)}
$$
Let $\bw = (w_1,\ldots,w_n) \in EV_L$, so that $M\bw = a \bw$ for some diagonal coefficient $a$ of $M$.
Thus we deduce that 
$$
D_\bw \bQ = 2 \sum_{a_i=a,a_j\neq a \& i\neq j} (a-a_j) w_i\;\rd v_j \;\bb_{(i,j)}
$$
Let $e$ be the dimension of the eigen-space $E_a$ of $L$ containing $\bw$. 
We can assume that $a = a_1$ and $w_1\neq 0$ and $w_j = 0$ once $j\geq 2$, so that 
$$
D_\bw \bQ = 2 \sum_{j\geq e+1} (a_1-a_j) w_1\; \rd v_j \; \bb_{(1,j)}
$$
Form this presentation we deduce the obvious following 
\begin{lemma}\label{lem:wedge-eigenspace}
Let $\bw$ be an eigen-vector of $M$ whose eigen-space has dimension $e$. Thus $\wdg^{n-e+1} D_\bw Q_L $ is identically 
null while $\wdg^{n-e} D_\bw Q_L$ is not identically null. Thus there $i<j$ such that $\omg_{(i,j)}(\bQ,\bw)$ is not 
a null $1$-form over $\Rn$.
\end{lemma}
\begin{proof}
It is indeed a property on $Q_L$, and computing with $\bQ$ provides the answer.
\end{proof}
To follow on Lemma \ref{lem:elementary-equal} and justify the material presented in Section \ref{section:SVEPA} 
we deduce
\begin{lemma}\label{lem:Qij-generate}
The coefficients $Q_{i,j}\in S^2F^\vee$ of the mapping $Q_L$ generates the ideal
of the polynomials over $F$ vanishing along $EV_L$.
\end{lemma}
%
%
%
%
%
%
%
%
%
%
%
%
%
%
%
%
%                ***********************************************************************
%
%
%
%
%
%
%
%
%
\section{The Essential Lemma}\label{section:TKL}
We present here a result of technical nature which is the substance of the structure of real analytic 
families of normal operators we wish to exhibit. Having in mind the material presented in Section 
\ref{section:SVEPA} and Section \ref{section:SLAMA}, will provide hints about how to use Lemma 
\ref{lem:essential} and Corollary \ref{cor:essential} below.

\bigskip
Let $\cU$ be a connected open subset of $\R^u$ and let $J:\cU \to H_2[X]^d$ be a real analytic
mapping with values in the quadratic polynomials over $\K^n$, where $d$ is some positive integer number. 
Let write $J = (J_1,\ldots,J_n)$.

Let $J_\bx$ be the ideal of $\K[X]$ generated by the quadratic polynomials
$(J_i(\bx))_{i=1,\ldots,d}$. Let $J_{\bx,2} := J_\bx\cap H_2[X]$ the $\K$-vector subspace of
$H_2[X]$ generated by $J_\bx$.

We further assume that the mapping $J$ satisfies the following hypotheses:

\smallskip
For each $\bx\in\cU$, the ideal $J_\bx$ is reduced and its vanishing locus is a bouquet of vector subspaces of $\K^n$,
$$
E_\bx := V(J_\bx) = \cup_{i=1}^{s_\bx} E_{i,\bx}
$$ 
such that 

(i) there exists $s\geq 1$ such that $s = s_\bx$ for all $\bx\in \cU$.

(ii) For each $\bx\in\cU$, the subspaces $(E_{i,\bx})_{i=1}^s$ are in direct sum.

(iii) For each $\bx\in \cU$, let $e_{i,\bx} := \dim E_{i,\bx} \geq 1$. 
Let $\be_\bx := (e_{1,\bx},\ldots,e_{s,\bx})$ be the $s$-tuple of the ordered dimensions of the sub-spaces $E_{i,\bx}$
(after a possible permutation on the indices, which \`a-priori depends on $\bx$).
The mapping $\bx \to \be_\bx$ is constant equal to $\be = (e_1,\ldots,e_s)$ with $e_i \geq 1$ and $e_1+\ldots+e_s = n$.

\medskip
The subset of $F :=\cU\times\K^n$ defined as 
$$
E := \cup_{\bx\in \cU}\,\{\bx\}\times V(J_\bx) = \cup_{\bx\in \cU} \,\{\bx\}\times E_\bx
$$
is a real analytic subset of $F$ and is called a \em bundle of bouquets of type $\be$.\em 

Let $(E_\bx)^\cdot$ be the pointed bouquet 
$E_\bx\setminus 0$. The subset $E^\cdot := \cup_{\bx\in\cU} \; \{\bx\}\times (E_\bx)^\cdot$ of $F^\cdot$ is 
the pointed cone bundle over $Z$ and is a real analytic subset of $F^\cdot$.

\medskip
The next result will be referred later as the \em Essential Lemma, \em
since everything we need to do is to produce a situation where it can be used.
\begin{lemma}\label{lem:essential}
The subset $E^\cdot$ is a real analytic submanifold of $F^\cdot$.
\end{lemma}
\begin{proof}
By the results of Section \ref{section:SVEPA}, let $d_J$ be the dimension of any of the vector-subspaces $J_{\bx,2}
\subset H_2[X]$, for all $\bx\in\cU$. We will re-use some of the notations of Section \ref{section:SVEPA} below.

Let $\by$ be a given point in $\cU$.
The ideal $J_\by$ is generated by exactly $d_J$ quadratic polynomials, which are $\K$-linear combinations
of $J_1(\by),\ldots,J_d(\by)$. Up to a linear change in $\K^n$, we can assume that for each $i=1,\ldots,s$
we have
$$
E_{i,\by} = \cap_{j\notin n_\be^i} \{\bu = (u_1,\ldots,u_n) \, : \, u_j = 0\}.
$$
The generators of $J_\by$ can be chosen to be the polynomials 
$c_{k,l} (\by)(X) = X_lX_l$ for $k,l$ as in  \ref{lem:elementary}, in such a way that 
each mapping $\bx \to c_{k,l}(\bx)(X) \in H_2[X]$ is real analytic over $\cU$ and for each $\bx$
gives a quadratic polynomial of $J_\bx$. Since these polynomials generate $J_\by$,
the family $(c_{k,l}(\bx))_{k,l}$ generates $J_\bx$ for $\bx$ nearby $\by$. Since the result we are looking for
is local at $\by$ we can assume that these polynomials generates $J_\bx$ for all $\bx \in \cU$. Thus, for $\bu\in \K^n$,
we can write 
$$
c_{k,l}(\bx)(\bu) = u_ku_l + \sum_{1\leq i<j\leq n} c_{k,l}^{i,j}(\bx) u_iu_j 
%%%%= u_ku_l + C_{k,l}(\bx)(\bu)
$$ 
where each function $\bx \to  c_{k,l}^{i,j}(\bx)$ is real analytic and vanishes at $\by$.
For any $\bu\in E_\by$, we have $c_{k,l}(\by)(\bu) = 0$. 
Let $1\leq i \leq s$ be given. For $\bu$ in $E_{i,\by}$, we do find  
\begin{center}
\begin{tabular}{rcl}
if $k,l \notin n_\be^i$ and $1\leq m \leq n$ & then & $\dd_{u_m} c_{k,l}(\by)(\bu) = 0$ \\
if $k,m\in n_\be^i$ and $l\notin n_\be^i$ & then & $\dd_{u_m}  c_{k,l}(\by)(\bu) = 0$ \\
if $k\in n_\be^i$ and $l,m\notin n_\be^i$ & then & $\dd_{u_m} c_{k,l} (\by)(\bu) = \delta_{l,m}u_k$
\end{tabular}
\end{center}
In particular we see that for each $k\in n_\be^i$
$$
\omg_k(\bu;\by) := \wdg_{l=e_i+1}^n D_\bu c_{k,l}(\by) = u_k^{n-e_i} \wdg_{l=e_i+1}^n \rd u_l.
$$
Taking $\bu \in E_\by \setminus 0$, there is $1\leq i \leq s$ such that $\bu\in E_{i,\by}$ and
there exists $k \in n_\be^i$ such that $\omg_k(\bu;\by) \neq 0$.

\smallskip
Let us consider the real analytic mapping 
$$
\begin{array}{rccl}
C^e \;\; : & F & \to &\K^{d_J} \\
 &(\bx;\bu) & \to & (c_{k,l}(\bx)(\bu))_{k,l}
\end{array}
$$
so that $(C^e)^{-1}(\{0\}^{d_J}) = E$.
Given $1 \leq r \leq n$, let $M_r$ be the ideal of $\cO_\cU[X]$ generated by the $(n-r)\times (n-r)$ minors of 
the $n\times d_J$ matrix $[\dd_{u_m}C^e]=[\dd_{u_m}c_{k,l}]_{m,(k,l)}$. 
Let $V(M_r) \subset F$ be the vanishing locus of this ideal and let $V(M_r)^\cdot$ be $V(M_r)\cap F^\cdot$.

\smallskip
Let $\be =(e_1,\ldots,e_s)$ be the $s$-tuple of the ordered dimensions of the bouquet-bundle $E$.
Let $e\in \{e_1,\ldots,e_s\}$. Let us consider the following subsets 
$$
E_\bx^{(e)} := \cup_{j\,:\,e_j = e} E_{\bx,j}\;\; {\rm and}\;\;
E^{(e)} := \cup_{\bx\in\cU}\; \{\bx\} \times E_\bx^{(e)}.
$$
In particular we see that 
$$
\cup_{k\leq e} (E_\bx^{(k)})^\cdot = E^\cdot \setminus V(M_e)^\cdot
$$
We deduce that each $E^{(r)}$ is a closed semi-analytic set of $F$
and is, if not empty, also a bouquet bundle of some type $(r,\ldots,r)$
where there are as many $r$ as the number of subspace in a(ny) fiber
$E_\bx^{(r)}$. Thus it is of dimension $u+r$ and of local dimension $u+r$ at each 
of its point. The local computation above producing the $(n-r)$-forms
$\omg_k$ shows that the projection $(E^{(r)})^\cdot$ onto $\cU$ is a submersion,
so that $(E^{(r)})^\cdot$ is a real analytic submanifold of dimension $u+r$ of $F^\cdot$.
\end{proof}
Since $E^\cdot$ is a real analytic submanifold of $F^\cdot$, we deduce that $Z := \bP E$
is closed real analytic submanifold of of $\bP F$.
Let $e \in \{e_1,\ldots,e_s\}$, and let $Z^{(e)} := \bP E^{(e)}$, which is a real analytic 
submanifold of $\bP F$ whose fibers consists of a (disjoint) union of $d_e := \#\{1\leq j \leq s\, : \, e_j = e\}$
linearly embedded projective spaces $\bP(\K^e)$. 
We deduce that, $Z$ has exactly $s$ connected components, namely 
$$
Z := \cP_1\sqcup\ldots\sqcup\cP_s
$$
such that each $\cP_i$ has dimension $u+e_i-1$. Each $\cP_i$ is a real analytic submanifold of $\bP F$ such that for 
each $\bx \in \cU$, the fiber $\cP_{i,\bx}$ is the projective space of 
one of the subspaces of the bouquet of $V(J_\bx) \subset F_\bx$ of dimension $e_i$. 
Thus we obtain the following
\begin{corollary}\label{cor:essential}
For each $i$, the projection $\bP F\to\cU$ induces a submersion $\cP_i \to \cU$.
More precisely each $\cP_i$ is a connected real analytic sub-bundle of $\bP F$, 
whose fiber is $\bP(\K^{e_i})\subset \bP(\K^n)$.
In other words $Cone(\cP_i)\subset F$, the cone-bundle over $\cP_i$, is a real analytic vector sub-bundle 
of $F$ over $\cU$.
\end{corollary}
\begin{proof}
The projection of $Z$ onto $\cU$ is a submersion 
as a result of the local computations done in the proof of Lemma \ref{lem:essential}, 
since $E$ is the $\K$-cone bundle over $Z$.
\end{proof}
%
%
%
%
%
%
%
%
%
%
%
%
%
%
%
%
%                ***********************************************************************
%
%
%
%
%
%
%
%
%
\section{Reduced model}\label{section:RM}
Let $F$ be the trivial $\K$-vector bundle $\cU\times \K^n$ over a connected open subset $\cU$ of $\R^u$.
We assume that $F$ is equipped with the Euclidean/Hermitian structure in each of its fiber.
Let $Hom(F,F) = \cU\times Hom(\K^n,\K^n)$ be the $\K$-vector bundle over $\cU$ of $\K$-endomorphisms of $F$.

\medskip
The tangent bundle $T F$ decomposes as the direct sum 
$$
T^h F\oplus T^v F
$$ 
of $T^hF$, the \em horizontal vector sub-bundle, \em with $T^vF$, the \em vertical vector sub-bundle, \em both 
being real analytic $\K$-vector bundles over $F$, where, for $\bp =(\by,\bw)$, the respective fibers
are defined as $T_\bp^h F := T_\by \cU$, and $T_\bp^v F := T_\bw F_\by$.
Of course we also find that 
$$
T\dot F = T^h \dot F\oplus T^v \dot F.
$$
A \em normal operator over $F$ \em is a real analytic section of $L: \cU \to Hom(F,F)$ which 
is normal in every fiber of $F$. 
We will consider, without change of notation, the restriction of $L$ to the pointed vector 
bundle $\dot F$.

\smallskip
Given a basis of $\K^n$, let $\cI_L$ be the $\cO_\cU$-ideal of coefficients of the matrix $M$ of $L$ in the given basis. 
Since it is defined by global sections it is a coherent $\cO_\cU$-ideal.
The vanishing locus of $\cI_L$ is the \em zero locus $V_L$ of $L$, \em that is the set of points $\bx\in \cU$ 
where $L(\bx)$ is the null operator. For the time being we will work under the following

\medskip\noindent 
{\bf Hypothesis 0.} \em Assume that $L$ is normal if $\K=\C$ and symmetric if $\K = \R$. \em

\smallskip
In order to postpone the use of blowings-up we demand

\medskip\noindent 
{\bf Hypothesis I.} \em Assume that $L$ is not a multiple of the identity. \em

\smallskip
As previously, we define the real analytic mapping of 
$\K$-vector bundles over $\cU$
$$
\begin{array}{rccl}
Q_L^e \;\; : & F & \to & \wdg^2F \\
 &(\bx;\bv) & \to & (\bx;Q_L(\bx;\bv) := L(\bx)\bv\wdg \bv)
\end{array}
$$
We will still write $Q_L^e$ and $Q_L$ for their restriction to $\dot F$.

\medskip
Let  $EV_L$ be \em the pointed eigen-bouquet bundle \em associated with the normal operator $L$
defined as 
$$
EV_L:=\{(\bx;\bv)\in \dot F: \; Q_L(\bx;\bv) = 0 \}. 
$$
Each fiber $EV_{L,\by}$ over $\by\in\cU$ is the (pointed) eigen-bouquet of $L(\by)$. 
Let $s(\by)$ be the number of distinct eigen-spaces of $L(\by)$.
Let $s_L := \max_\cU s(\by) \leq n$. Hypothesis I implies that $s_L\geq 2$.
Let 
$$
D_L := \{\by\in\cU: \; s(\by) < s_L\}
$$
which is the discriminant locus of the reduced characteristic polynomial of $L$
(also called generalized discriminant),
thus a closed real analytic subset of $\cU$ of positive codimension if not empty.
The function $s(\by)$ is then Zariski-analytic semi-continuous, that is constant 
and equal to $s_L$ over $\cU\setminus D_L$.

\smallskip
Let $(\cU^i)_{i\in I_L}$ be the connected components of $\cU\setminus D_L$, which are locally 
finitely many. 
Let $i\in I_L$ be given. For any $\bx\in\cU^i$ the operator $L(\bx)$ has 
exactly $s_L$ eigen-spaces $E_{j,\bx}^i$ for $j=1,\ldots,s_L$. 

\smallskip
We recall the following well known fact:
\begin{proposition}\label{prop:vector-bundle}
For each $i\in I_L$ and each $j = 1,\ldots,s_L$ the subset $E_j^i := \cup_{\bx\in \cU^i} E_{j,\bx}^i$ of $F^i:=F|_{\cU^i}$
is a real analytic $\K$-vector bundle over $\cU^i$.
\end{proposition}
\begin{proof}
Assume that $i$ is fixed. Let $\by$ be a point of $\cU^i$.

\smallskip
By Lemma \ref{lem:Qij-generate}, the quadratic polynomials $Q_{i,j}(\by)$, components of the mapping $Q_L$,
generate the ideal $J_\by$ of the polynomials vanishing over $EV_{L,\by} \cup 0 \subset F_\by$.
Thus $J_\by$ is reduced and we conclude using the Essential Lemma \ref{lem:essential} and Corollary 
\ref{cor:essential}.
\end{proof}
\begin{remark}
The proof of Proposition \ref{prop:vector-bundle} did not use any regularity property of the eigen-values.
\end{remark}
%%%%
%
Since $F$ is the trivial vector bundle $\cU\times\K^n$, 
we also find that 
\begin{equation}\label{eq:s2f-sf}
S^2F^\vee = \cU\times H_2[X] \;\; {\rm and } \;\; SF^\vee = \cU\times \K[X].
\end{equation}

The ordered $s_L$-tuple of dimensions of the eigen-spaces of $L$ is constant outside $D_L$.
Let us denote $\be_L \in S_{s_L}^n$ this ordered $s_L$-tuple of dimensions.

\medskip
For $\by\in\cU$, let 
$$
\Sgm_L^2:=\{(\bx,\bu,P)\in F\oplus S^2F^\vee \,:\, L(\bx)\bu \wdg\bu = 0 
\;\; {\rm and }\;\; P(\bx,\bu) = 0\}, 
$$
be the bundle of quadratic polynomials over $F$ vanishing along the (pointed) 
eigen-bouquet bundle $EV_L$.
Let $\pi_{S^2F^\vee}$ be the projection of $F\oplus S^2F^\vee$ over $S^2F^\vee$. Let
$$
J_{\by,2}:= \pi_{S^2F^\vee}((\Sgm_L^2)_\by)
$$
be the vector subspace of $S^2F^\vee_\by$ of the quadratic polynomials lying 
in $J_\by$, as was done in Section \ref{section:SVEPA}.
Let $d(\by) = d_2(J_{\by,2})$ be the dimension of $J_{\by,2}$ as in 
Section \ref{section:SVEPA}.
Following Lemma \ref{lem:elementary}, $d(\by)$ is the number of minimal generators of $J_\by$.

The vector subspace $J_{\by,2} \subset (S^2F^\vee)_\by = H_2[X] \simeq\K^{\frac{n(n+1)}{2}}$, 
for $\by$ outside of $D_L$, is of constant dimension 
$d_L \leq \frac{n(n-1)}{2}$ since its zero locus is the union of the eigen-spaces 
of $L(\by)$ and the ordered $s_L$-tuples
of the dimensions of these eigen-spaces is constant over $\cU\setminus D_L$ and equal to $\be_L$. 
At each point $\by$ of $D_L$ we find that $d(\by) < d_L$ by Corollary \ref{cor:elementary}.

Once an orthonormal/unitary basis of $\K^n$ is given, we obtain an explicit expression
$\bQ$ of the mapping $Q_L$. 
Its components, the mappings $\by \to Q_{i,j}(\by)$ are real analytic sections of $S^2F^\vee$, for any pair $1\leq i<j\leq n$.
Let $\cA_L$ be the coherent $\cO_\cU$-module of sections of the vector bundle (in the literal sense of Section 2) 
$\cup_{\by\in\cU} J_{\by,2} \subset S^2F^\vee$. It is generated by the $Q_{i,j}$ and is locally free of rank $d_L$. 
Let us also consider the following $\cO_\cU$-module 
$$
\cM_L := \wedge^{d_L} \cA_L.
$$ 
It is coherent and generated by 
$\wdg_{k=1}^{d_L} Q_{i_k,j_k}$ taken over all possible $d_L$-tuples of vectors $Q_{i,j}(\by)$. 
Its $\cO_\cU$-ideal of coefficients $\cF_L$ is the maximal Fitting ideal of $\cA_L$ with 
co-support exactly $D_L$, generated by the $d_L\times d_L$-minors of the 
$\frac{n(n-1)}{2}\times \frac{n(n+1)}{2}$-matrix formed by the sections
$Q_{i,j}$.  

Since $\wedge^2 F$ embeds canonically into $S^2F^\vee$ (as skew-symmetric forms),
the mapping 
$$
\wedge^{d_L} Q_L : \cU \to \wedge^{d_L}S^2F^\vee
$$
is a well defined real analytic section of $\wedge^{d_L}S^2F^\vee$, and its coefficients generate 
the module $\cM_L$. At each $\bx\notin D_L$ the (simple) $d_L$-vector $\wedge^{d_L} Q_L(\bx)$ represents the 
vector subspace $J_{\bx,2}$. This gives rise to the following real analytic section
$$
\bp_L :\cU\setminus D_L: \to \bP(\wedge^{d_L} S^2F^\vee)
$$ 
which factors through the Pl\"ucker embedding 
$$
\bG(d_L,S^2F^\vee_\by)\to \bP(\wedge^{d_L}S^2F^\vee).
$$
Along $D_L$ we get $(\wedge^{d_L}Q_L)|_{D_L} =0$. 

\medskip
Thus up to "principal-izing" the ideal $\cF_L$ we assume for the time being the following,

\medskip\noindent 
{\bf Hypothesis II.} \em The $\cO_\cU$-ideal $\cF_L$ is principal. \em

\smallskip
Let $\cM_L'$ be the $\cO_\cU$-module $\cF_L^{-1}\cdot\cM_L$.
It is a locally $\cO_\cU$-free module of rank $1$ with empty co-support. 
Equivalently the mapping $\bp_L$ extends analytically 
to the whole $\cU$ since we can divide by the local generator of $\cF_L$ at every point. 
Lifting the mapping $\bp_L$ to the (trivial) real analytic universal bundle of $\bG(d_L,S^2F^\vee)$,
provides a real analytic correspondence of points $\bx\in\cU$ to a $d_L$-dimensional subspace $J_{\bx,2}'$ of 
$S^2F^\vee_\bx = H_2[X]$ such that at any point $\by\notin D_L$ we know that 
$$
J_{\by,2} = J_{\by,2}'.
$$
In other words the module $\cM_L'$ is the "cone" over $\bp_L$.

\smallskip\noindent
Let $\bA_L$ be the (non-trivial) real analytic vector sub-bundle of $S^2F^\vee$ 
lifting $\bp_L$.
Let $\bx \to (A_\tau'(\bx))_{\tau=1,\ldots,d_L}$ be a local real analytic frame 
at $\by\in\cU$ of $\bA_L$. Suppose it is defined over a neighbourhood $\cV$ of $\by$.
Let $Q_\tau'(\bx;\bv)$ be the quadratic polynomial over $F_\bx$ defined by the vector $A_\tau'(\bx)$.
Thus 
$$
\cap_\tau \,\{(\bx;\bv)\in \dot{F}|_\cV \,:\; Q_\tau'(\bx;\bv) = 0 \} \subset EV_L \cap \pi^{-1}(\cV)
$$
and if $\cV\cap D_L$ is empty then
$$
\cap_\tau\,\{(\bx;\bv)\in \dot{F}|_\cV \,:\; Q_\tau'(\bx;\bv) = 0 \} = EV_L \cap \pi^{-1}(\cV)
$$
Let $J_\by'$ be the ideal of $SF^\vee_{\by} = \K[X]$ generated by the quadratic polynomials $Q_\tau'(\by)$, in 
other words generated by $J_{\by,2}'$.
Let $(\bx_k)_k$ be a sequence of $\cU\setminus D_L$ converging to $\by\in D_L$.
Let $E_{1,k},\ldots,E_{s_L,k}$ be the eigen-spaces of $L(\bx_k)$. Up 
to taking a sub-sequence, we can assume that each $(E_{i,k})_k$ converges (in the right Grassmannian) 
to $E_i \subset F_\by$. In particular the subspaces $E_1,\ldots,E_{S_L}$ are in orthogonal direct sum
and the restriction of $L(\by)$ to any $E_i$ is a multiple of the identity mapping of $E_i$.
By continuity along $\cV$, all the functions $Q_\tau'$ must vanish along the union  $E_\by :=\cup_{i=1}^{s_L} E_i$,
that is: the vanishing locus of $J_\by'$ contains $E_\by$.
The ideal $K_\by$ of quadratic polynomials vanishing along $E_\by$ generates (see Section \ref{section:SVEPA}) 
a vector subspace $K_{\by,2}$ of $S^2F^\vee_\by$ which following Lemma \ref{lem:elementary} is
of dimension $d_L$, since the $s_L$-tuple of the ordered dimensions of $E_1,\ldots,E_{s_L}$ is obviously $\be_L$. 
The reduced ideal $K_\by$ contains $J_\by'$ and by Lemma \ref{lem:elementary-equal} we deduce that $J_\by' = K_\by$.

\smallskip
The following result, although not necessary in order to have a complete proof, 
shows that at every point of the discriminant locus, in the current post-resolution of singularities 
setting we are working with, there exists a unique limit eigen-bouquet. 
\begin{lemma}\label{lem:limit-unique}
Let $\by$ be a point of $D_L$. 

\smallskip 
(i) The zero locus of $J_\by'$ consists 
of the union of pairwise orthogonal subspaces $G_1,\ldots,G_{s_L}$ which
are in direct sum. 

\smallskip
(ii) Let $\cU_i$ be a connected component of $\cU \setminus D_L$ whose closure contains $\by$.
Let $E_{j,\bx}^i$ be the eigen-space of $L(\bx)|\cU^i$.
Let $(\bx_k)_k$ be any sequence of points of $\cU_i$ converging to $\by$. 
Assume that each $E_{j,\bx_k}^i$ converges to $P_j$. There exists $\ve$ a permutation of 
$\{1,\ldots,s_L\}$, independent of the sequence $(\bx_k)_k$ taken, such that for 
the set of corresponding limit $(P_j)_j$ at $\by$ we have $P_j = G_{\ve(j)}$. 

\em In other words the \em real analytic vector bundles $E_j^i$, for $j=1,\ldots,s_L$ over $\cU_i$ extends as 
vector bundles of constant rank over $\clos(\cU_i)$. 
\end{lemma}
\begin{proof}
Point (i) comes from the fact $J_\by' = K_\by$.

\smallskip\noindent
For any sequence like in (ii), the zero locus of $J_\by'$ 
must vanish along $\cup_{i=1}^{s_L} P_i$. The uniqueness of the permutation $\ve$ 
comes from the pairwise orthogonality of the vector sub-bundles $(E_j^i)_{j=1,\ldots,s_L}$.
\end{proof}

\medskip
So far we have produced the real analytic vector sub-bundle $\bA_L$ of $S^2F^\vee$ 
of rank $d_L$ lifting the mapping $\bp_L$. Thus we have produced a real analytic family 
of quadratic ideals $(J_\by')_{\by\in\cU}$ of $SF^\vee$.
For each $\by$, the zero locus $EV_{L,\by}'\cup \{0\} \subset \dot{F}_\by\cup \{0\}$ of $J_\by'$ in 
$F_\by = \K^n$ consists of a bouquet of vector subspaces of $F_\by$, 
in orthogonal direct sum contained in the eigen-bouquet $EV_{L,\by}$ of $L(\by)$. We define 
$$
EV_L' := \bigcup_{\by\in\cU} EV_{L,\by}' \subset EV_L \subset \dot{F}
$$
the "bundle-d" union taken over $\cU$ of the (pointed) vanishing loci of the ideals $J_\by'$.
We call $EV_L'$ the \em reduced \em(pointed)  \em eigen-bouquet bundle of 
$L$, \em
which is a closed real analytic subset of $EV_L$, thus of $\dot{F}$.
\begin{lemma}\label{lem:regular}
The subset $EV_L'$ is real analytic submanifold of $\dot F$.
\end{lemma}
\begin{proof}
Obtaining the ideals $(J_\bx')_{\bx\in\cU}$ was to be exactly in a situation to 
apply the Essential Lemma \ref{lem:essential}, which are in.
\end{proof}
Since $EV_L'$ is fiber-wise a (pointed) bouquet of 
subspaces of $F$ and a closed real analytic subset and a sub-manifold of $\dot{F}$, 
it defines the real analytic projective sub-bundle  
$$
\bP EV_L' \subset \bP F,
$$
consisting of $s_L$ connected components, namely 
$$
\cP_1\sqcup\ldots\sqcup\cP_{s_L}
$$
such that each $\cP_i$ has dimension $u+e_i-1$, where $(e_1,\ldots,e_{s_L}) = \be_L$. 
Each $\cP_i$ is a real analytic submanifold of $\bP F$ and a real analytic projective sub-bundle with fiber $\bP(\K^{e_i})$, 
and each $\cP_{i,\bx}$ is the projective space of one of the component of the bouquet of subspaces of 
$V(J_\bx) \subset F_\bx$ of dimension $e_i$, and is an eigen-space of $L(\bx)$. 
The proof of our main result under the conditions of this section ends as a consequence of 
Corollary \ref{cor:essential}:
\begin{corollary}\label{cor:regular}
(i) For each $i$, the projection $\bP F\to\cU$ induces a submersion $\cP_i \to \cU$.
More precisely, each $\cP_i$ is a connected real analytic sub-bundle of $\bP F$, 
whose fiber is $\bP(\K^{e_i})\subset \bP(\K^n)$.
In other words $\cE_i:=Cone(\cP_i)\subset F$, the cone-bundle over $\cP_i$, is a real analytic vector sub-bundle 
of $F$ over $\cU$.

\smallskip
(ii) Under the hypotheses 0 and I and II, any point $\by$ admits a neighbourhood 
$\cW$ over which exist $n$ real analytic sections of $F$ which form
an orthonormal/unitary frame of $F$ at every $\bx$ of $\cW$ consisting of eigen-vectors of $L$.

\smallskip
(iii) The eigen-values can be locally chosen real analytic.
\end{corollary}
\begin{proof}
Point (i) is the statement of Corollary \ref{cor:essential}  which can be applied here.

\smallskip
Points (ii) and (iii) are obvious since the real analytic vector sub-bundles $\cE_i$ are pairwise
orthogonal, in direct sum and contained point-wise in the eigen-spaces of $L$.
\end{proof}
%
%
%
%
%
%
%
%
%
%
%
%
%
%
%
%                ***********************************************************************
%
%
%
%
%
%
%
%
%
\section{General Model and Hypothesis 0}\label{section:GEM}
We are pursuing what was started in Section \ref{section:RM} but dropping hypotheses (I) and (II).
Our work here will be, by means of blowings-up, to recreate a situation
where we can work under these dropped hypotheses. 

\medskip
We fix once for all an orthonormal basis $\cB = (\bb_1,\ldots,\bb_n)$ of $\K^n$.
We recall that for $i<j$ the notations $\bb_{(i,j)}$ stands for the $2$-vector $\bb_i\wdg\bb_j\in\wdg^2\K^n$.

\smallskip
The normal operators $L(\bx)$ are given by complex normal (or real symmetric) matrices $M(\bx)$, in the basis $\cB$, whose 
coefficients are real analytic functions over $\cU$, thus global sections of $\cO_\cU$.
In order to have some work to do here we can assume that

\medskip\noindent 
{\bf Hypothesis.} \em The operator $L$ is not a multiple of the identity. \em

\medskip
The ideal $\cI_L$ of the coefficients of $L(\bx)$ is generated by finitely many 
global sections, thus is $\cO_\cU$-coherent. If $\chi_L (\bx;T)$ denotes the reduced characteristic 
polynomial of the normal matrix $M(\bx)$, then its discriminant locus (as a subset of $\cU$) 
is just $D_L$. Under Hypothesis 0, we have $V_L \subset D_L$ is of codimension one or more in $\cU$.

\medskip
Let, as before, be $Q: F\to \wedge^2 F$ defined as 
$$
Q(\bx;\bv) = M(\bx)\bv\wedge\bv = \sum_{i<j} Q_{i,j} \bb_{(i,j)}.
$$
Let again $\cA_L$ be the $\cO_\cU$-sub-module of sections of $\wdg^{d_L} S^2F^\vee = \cU \times \wdg^{d_L} H_2[X]$ 
generated by the  global sections $(Q_{i,j})_{1\leq i<j\leq n}$, thus is $\cO_\cU$-coherent.
Let $\cM_L := \wedge^{d_L}\cA_L$. The ideal $\cF_L$ of coefficients of $\cM_L$, is 
the maximal Fitting ideal of $\cA_L$, generated by the $d_L\times d_L$-minors of 
the $\frac{n(n-1)}{2}\times \frac{n(n+1)}{2}$-matrices of the vectors $Q_{i,j}$s, thus is also $\cO_\cU$-coherent
with co-support exactly $D_L$.
\begin{lemma}\label{lem:GEM-1}
There exists $\sgm:(\wtcU,\wtcE) \to (\cU,D_L)$ a locally finite composition of 
geometrically admissible blowings-up such that 

\smallskip
(i) the $\cO_\wtcU$-ideal $\sgm^*\cF_L$ is principal and monomial in the simple normal crossing divisor $\wtcE$.

\smallskip
(ii) The $\cO_\wtcU$-module $\cM_L' := (\sgm^*\cF_L)^{-1}\sgm^*\cM_L$ is locally free of rank $1$ with empty co-support.
\end{lemma}
\begin{proof}
The first point is a direct application of the embedded resolution of ideals \cite{Hir,BiMi1}.

The module $\cM_L'$ has empty co-support by definition of $\cM_L$ and $\cF_L$.
It is locally free as was already seen in Section \ref{section:RM}.
\end{proof}

\medskip
We can now state the main result given the context of this section.
\begin{theorem}\label{thm:main-local}
Let $L:\cU \to F = \cU\times \K^n$ be a real analytic normal operator (real symmetric if $\K = \R$), where $\cU$ is a 
connected subset of $\R^u$. Assume that $L$ is not a multiple of identity.

There exists a locally finite composition of geometrically admissible blowings-up $\sgm:(\wtcU,\wtcE) \to (\cU,D_L)$
such that for any $\tldby \in \wtcU$, there exists $\wtcV$ a neighbourhood of $\tldby$ such that 
there exists real analytic sections of $(\sgm^* F)^\cdot$ which form an orthonormal/unitary basis of $L(\sgm(\bx))$ 
for each $\bx \in \wtcV$.
\end{theorem}
\begin{proof}
Let $\sgm$ be that of Lemma \ref{lem:GEM-1}.

From here on we proceed as in Section \ref{section:RM}. Indeed all we need to do from now 
is of local nature, thus was done there, and will work as well here,
since $\cM_L'$ has the exact same properties.
By Corollary \ref{cor:regular}, for each $\tldby\in\wtcU$ there exists 
a neighbourhood $\wtcV$ of $\tldby$ such that 
there exists $R_1,\ldots,R_{s_L}$ real analytic vector sub-bundles of $\sgm^* F|_\wtcV$
which are contained in $\sgm^*EV_L$ and such that they form an orthogonal
direct sum of $\sgm^* F|_\wtcV$. Shrinking $\wtcV$ if necessary so that it is a chart domain of 
each $R_i$, those vector bundles are all real analytically trivial over $\wtcV$, thus the statement.
\end{proof}
Now we deduce the following
\begin{corollary}\label{cor:eigen-value-regular}
For any $\wtby \in \wtcU$, there exists $\wtcV$ a neighbourhood of $\tldby$ such that 
the eigen-values of $\wtbx\to L(\sgm(\wtbx))$ are real analytic.
\end{corollary}
\begin{proof}
Just evaluate the mapping $(\wtbx,\bv) \to \frac{L(\sgm(\wtbx))\bv}{|\bv|}$ over each $R_i\setminus 0$.
\end{proof}
\begin{remark}\label{rk:one-dim} 
If the parameter space were to be of dimension $1$, that is $u=1$, the discriminant locus $D_L$ is discrete.
Up to locally dividing $\cF_L$ by a real analytic function germ vanishing at a discriminant point $\bt_0$, 
our presentation recover the results of Rellich and Kato, with little work.
\end{remark}
%
%
%
%
%
%
%
%
%
%
%
%
%
%
%
%                ***********************************************************************
%
%
%
%
%
%
%
%
%
\section{General Case}\label{section:GC}
Let $N$ be a connected real analytic manifold of dimension $u\geq 1$.

Let $F$ be a real analytic $\K$-vector bundle over $N$ equipped with a real analytic fiber metric
structure, 
that is endowed with a real analytic section $N \to S^2F^\vee$, 
which is positive definite ($\K=\R)$ or Hermitian ($\K=\C$) over every fiber of $F$. 
The local model of $F$ is $\R^u\times\K^n$ but with transitions taken into 
the orthogonal group $O(\R^n)$ or the unitary group $U(\C^n)$.

\medskip
Let $L$ be a real analytic normal operator $F\to F$, which we suppose symmetric if $\K =\R$.
Since $N$ is connected, the set of points where the number of distinct eigen-values 
of $L$ is not locally maximal is a real analytic subset $D_L$ of $N$ of positive
codimension. Let $s_L$ be this maximal number of distinct eigen-values.

Let $\cU$ be an open subset of $N$ over which $F$ is trivial. 
We can then choose an orthonormal basis $\cB^\cU$, and then do
as in Section \ref{section:GEM} for $F|_\cU$. 
Let $L_\cU$ be the restriction of $L$ to $F|_\cU$.
Let $Q_\cU$ and $Q_{(i,j),\cU}$ be $Q$ and $Q_{i,j}$ of Section \ref{section:GEM} 
for $L_\cU$.
In particular we can construct a coherent $\cO_\cU$-module $\cA_\cU := \cA_{L_\cU}$
which is locally free of rank $d_\cU$. 
Note that all these $d_\cU$ are equal to $d_L \geq 1$ since $D_L$ is nowhere dense in $N$.
Let define $\cM_\cU := \wedge^{d_L}\cA_\cU$
which is $\cO_\cU$-coherent.
% at any point outside $D_\cU := D_L\cap\cU$.
Let $\cF_\cU$ be the coherent $\cO_\cU$-ideal of coefficients of $\cM_\cU$, defined
globally over $\cU$, since generated by the $d_L\times d_L$-minors of the matrix of the $Q_{(i,j),\cU}$.

Let $\cW_1$ and $\cW_2$ be two open subsets of $N$ intersecting each other and over each of which 
exist respectively a basis $\cB_1$ and a basis $\cB_2$. 
Taking trivializations of $F$ over $\cW_1$ and over $\cW_2$ we see that $\cM_{\cW_1}|_{\cW_1\cap\cW_2}$
and $\cM_{\cW_2}|_{\cW_1\cap\cW_2}$ are isomorphic. 
Let $P$ be the orthogonal/unitary matrix passing from the basis over $\cW_1$ to the basis over $\cW_2$.
Thus at any point $\by\in\cW_1\cap\cW_2$, we check that $P$ induces a linear automorphism $P_2[n]$
over $H_2[X]$ such that the vector subspace of $S^2F^\vee_\by$ spanned by $(Q_{(i,j),\cW_1}(\by))_{i<j}$ is 
mapped onto the vector subspace spanned by $(Q_{(i,j),\cW_2}(\by))_{i<j}$.
In particular this implies that we can extend each $\cM_{\cW_i}$ as a $\cO_{\cW_1\cup\cW_2}$-module, 
which is still coherent.
Thus we produce a coherent $\cO_N$-module $\cM_L$ 
vanishing only over $D_L$. And its coefficient ideal $\cF_L$ is also $\cO_N$-coherent.

\medskip 
And we arrive at the main result of the paper:
\begin{theorem}\label{thm:main-global}
Let $L:F \to F$ be a real analytic  normal operator over a real analytic $\K$-vector bundle
$F$ of finite rank over a real analytic and connected manifold $N$, and equipped with a real analytic fiber 
metric. 

\smallskip
There exists a locally finite composition of geometrically admissible blowings-up $\sgm:(\wtN,\wtE) \to (N,D_L)$
such that for any $\tldby \in \wtN$, there exists $\wtcV$ a neighbourhood of $\tldby$ and
real analytic vector sub-bundles $R_1,\ldots,R_s$ of $\sgm^*F|_\wtcV$ such that 

(a) they are pair-wise orthogonal and everywhere in direct sum;

(b) The restriction of $L\circ\sgm$ to each $R_i$ is reduced: either it is a multiple
of identity, or $R_i$ is real of real rank 2, and the restriction of $L\circ\sgm$
is a fiber-wise a similitude. 
\end{theorem}
And we deduce the obvious
\begin{corollary}\label{cor:main-global}
For any $\tldby \in \wtN$, there exists $\wtcV$ a neighbourhood of $\tldby$ such that 
the eigen-values of $\wtbx\to L(\sgm(\wtbx))$ are real analytic.
\end{corollary}
\begin{proof}[Proof of Theorem \ref{thm:main-global}]
The proof of the real and non symmetric case will be dealt with in Section \ref{section:RC}

\medskip
First if $s_L = 1$, that is $L$ is a multiple of the identity operator then there is nothing to do.

\medskip\noindent
Assume that $s_L\geq 2$, so that $D_L$ is not empty. 

\smallskip\noindent
Assume also that $L$ is complex normal or real symmetric.

The ideal $\cF_L$ is $\cO_N$-coherent, there exists a locally finite composition of geometrically admissible 
blowings-up $\sgm:(\wtcU,\wtcE) \to (N,D_L)$ such that the ideal for any $\tldby \in \wtcU$, there exists $\wtcV$ 
a neighbourhood of $\tldby$ such that there exists real analytic sections of $(\sgm^* F)^\cdot$ which form an 
orthonormal/unitary basis of $L(\sgm(\wtbx))$ for each $\wtbx \in \wtcV$.

And the rest is as we did in Section \ref{section:RM} and the proof of Theorem \ref{thm:main-local}.

\medskip
As for the second point, it also works as in Corollary \ref{cor:eigen-value-regular}
\end{proof}
\begin{remark}
The reader will have noticed that we never uses the condition of monomialization of the ideals we desingularize. 
We just need to have there weak transforms principal.
\end{remark}
The next remark is inspired by our original motivations (see \cite{Gra}).
\begin{remark}\label{rmk:invertible-sheaf}
What we did here for a normal operator of $Hom(F,F)$ also holds true for a real analytic (coherent) invertible $\cO_N$-sheaf
of $Hom(F,F)$ locally generated by a normal operator.
In such a setting eigen-values are not objects attached to the invertible sheaf. 
\end{remark}
%
%
%
%
%
%
%
%
%
%
%
%
%
%                ***********************************************************************
%
%
%
%
%
%
%
%
%
\section{Real case}\label{section:RC}
We complete the proof of Theorem \ref{thm:main-global}, treating the real 
non-symmetric case. 
\begin{theorem}\label{thm:main-real}
Let $L:F \to F$ be a real analytic  normal operator over a real analytic $\R$-vector bundle
$F$ over a real analytic and connected manifold $N$, and equipped with a real analytic fiber Euclidean structure. 

\smallskip
(i) There exists a locally finite composition of geometrically admissible blowings-up $\sgm:(\wtN,\wtE) \to (\cU,D_L)$
such that for any $\tldby \in \wtN$, there exist $\wtcV$ a neighbourhood of $\tldby$ and
 real analytic sections of $\sgm^* F$ which form an orthonormal basis of characteristic 
vectors of $L(\sgm(\bx))$ for each $\bx \in \wtcV$.

\smallskip
(ii) For any $\tldby \in \wtN$, there exists $\wtcV$ a neighbourhood of $\tldby$ such that 
the eigen-values of $\wtbx\to L(\sgm(\wtbx))$ are real analytic.
\end{theorem}
\begin{proof}[Proof of Theorem \ref{thm:main-real}]
As introduced at the end of Section \ref{section:ANO}, let us write 
$$
L = A + B, \;\; {\rm where} \;\; A:=\frac{L+L^*}{2} \;\; {\rm and} \;\; 
B = \frac{L-L^*}{2}.
$$
It is sufficient to have the result for anti-symmetric operators $F\to F$, 
since it is already true for symmetric ones. 

\medskip
Assume the normal operator $L:F\to F$ is anti-symmetric, that is $L = B$.
The real analytic vector bundle $F\oplus F$ is equipped with the
product of the scalar product of each embedding $O\oplus F$ and $F\oplus O$,
where $O$ is the zero-section of $F$.
We construct, as in Section \ref{section:ANO}, a symmetric operator over $F\oplus F$ from $L$, namely 
\begin{center}
$
\begin{array}{ccccc}
B_2 &  : & F\oplus F & \to & F \oplus F \\
 & & \bu\oplus\bv & \to & -B\bv \oplus B\bu
\end{array}
$
\end{center}
Its kernel is the direct sum $\ker L \oplus \ker L$.
In order to avoid introducing new notations we can already assume
that $B_2$ satisfies Theorem \ref{thm:main-real} for symmetric operator.

Let $\by$ be any point of $N$. There exists a neighbourhood $\cV$ of $\by$
such that there exists a real analytic section 
$$
f : \cV\to Frame(F\oplus F)
$$
such that for each $\bx$ in $\cV$ the frame $f(\bx)$ of $F_\bx\oplus F_\bx$
is orthonormal and consists of eigen-vectors of $B_2(\bx)$. 
We can always assume that $F$ is trivial over $\cV$.

There exists a (maximal) real analytic vector sub-bundle $R_b$ of $(F\oplus F)|_\cV$
consisting only of eigen-vector of $B_2$ associated to the 
eigen-value $b\neq 0$ (outside $D_L$). 
%We assume that $R_b$ is not the kernel of $B_2|\cV$.
Thus $R_b$ is generated by $\bff_1,\ldots,\bff_{2d}$ sections of $F\oplus F$
over $\cV$ and vectors of the frame $f$.
Let $\bff_j= \bu_j\oplus\bv_j$ for pairs of orthogonal sections $\bu_j,\bv_j$ of $F$ and 
let $\bff_i' := \bu_i \oplus(-\bv_i)$. 
Following the classical material presented in Section \ref{section:ANO}, 
The sections $\bff_1',\ldots,\bff_{2d}'$ form an orthonormal basis of $R_{-b}$,
the eigen-space bundle of $B_2$ corresponding to $-b$. We observe that 
$\bu_i$ and $\bv_i$ never vanish over $\cV$. 
We can assume that $\bff_1',\ldots,\bff_{2d}'$ are also in the frame $f$.
By Lemma \ref{lem:plane-invariant}, the section $(-\bv_1)\oplus\bu_1$ 
lies in $R_b$ and is orthogonal to $\bff_1$. So
we can assume that $\bff_2 = (-\bv_1)\oplus\bu_1$ is a vector of the frame $f$.

Let $f_b$ be the free family of vectors of the frame $f$ generating $R_b$, and $f_{-b}$
be the free family of $f$ generating $R_{-b}$ (orthogonal to $f_b$).
Thus $f_b = \{\bff_1,\bff_2,f_b^1\}$ where $f_b^1$ is orthonormal and orthogonal 
to $\bff_1,\bff_2$, and $f_{-b} =\{\bff_1',\bff_2',f_{-b}^1\}$
where $f_{-b}^1$ is orthonormal and orthogonal 
to $\bff_1',\bff_2'$. 
Let $R_b^1$ (resp. $R_{-b}^1$) be the real analytic sub-bundle of $R_b$ (resp. $R_{-b}$) generated by $f_b^1$
(resp. $f_{-b}^1$).
Note that the plane $\R\bff_1 + \R\bff_2 = \R\bff_1' + \R\bff_2' \subset F\oplus F$ and $R_b^1,R_{-b}^1$ are both invariant 
by $B_2$.

A descending induction process produces a real orthogonal and maximal free family
$\bu_1,\bv_1,\ldots,\bu_d,\bv_d$ of $F$, such that
each plane $C_{b,j} := \R\bu_j+\R\bv_j \subset V$ is an ARCP of $B$ and 
the eigen-vector space of $B^2$ associated to $-b^2$ is the direct sum 
of the $C_{b,j}$ $j=1,\ldots,2$. 

And this concludes point (i) in the anti-symmetric case.

\medskip
The eigen-values of $B$ are exactly the $\pm\sqrt{-1}b_i$, which are already analytic.
\end{proof}
%
%
%
%
%
%
%
%
%
%
%
%
%
%
%
%
%
%
%
%
%                ***********************************************************************
%
%
%
%
%
%
%
%                
%
%
%
%
%
%
%
%
%
%
%
%
%

%
%
\end{document}